\documentclass[11pt]{amsart}
\usepackage[margin=1.35in]{geometry}
\usepackage{amscd,amsmath,amsxtra,amsthm,amssymb,stmaryrd,xr,mathrsfs,mathtools,enumerate,commath, comment}
\usepackage{stmaryrd}
\usepackage{xcolor}
\usepackage{commath}
\usepackage{comment}
\usepackage{tikz-cd}
\usepackage{longtable} 
\usepackage{pdflscape} 
\usepackage{booktabs}
\usepackage{hyperref}
\definecolor{vegasgold}{rgb}{0.77, 0.7, 0.35}
\definecolor{darkgoldenrod}{rgb}{0.72, 0.53, 0.04}
\definecolor{gold(metallic)}{rgb}{0.83, 0.69, 0.22}
\hypersetup{
 colorlinks=true,
 linkcolor=darkgoldenrod,
 filecolor=brown,      
 urlcolor=gold(metallic),
 citecolor=darkgoldenrod,
 pdftitle={Galois representations ramified at one prime and with suitably large image},
 }

\usepackage[all,cmtip]{xy}

\DeclareFontFamily{U}{wncy}{}
\DeclareFontShape{U}{wncy}{m}{n}{<->wncyr10}{}
\DeclareSymbolFont{mcy}{U}{wncy}{m}{n}
\DeclareMathSymbol{\Sh}{\mathord}{mcy}{"58}
\usepackage[T2A,T1]{fontenc}
\usepackage[OT2,T1]{fontenc}

\newtheorem{theorem}{Theorem}[section]
\newtheorem{lemma}[theorem]{Lemma}
\newtheorem{ass}[theorem]{Assumption}
\newtheorem*{theorem*}{Theorem}
\newtheorem*{ass*}{Assumption}
\newtheorem{definition}{Definition}[section]

\newtheorem{remark}[theorem]{Remark}

\newtheorem{proposition}[theorem]{Proposition}
\newtheorem{question}[theorem]{Question}

\newcommand{\Gp}{\operatorname{G}_{\{p\}}}

\newcommand{\g}{\mathfrak{g}}

\newcommand{\Z}{\mathbb{Z}}

\newcommand{\Q}{\mathbb{Q}}
\newcommand{\F}{\mathbb{F}}

\newcommand{\cC}{\mathcal{C}}

\newcommand{\A}{\mathcal{A}}

\newcommand{\op}[1]{\operatorname{#1}}

\newcommand\mtx[4] { \left( {\begin{array}{cc}
 #1 & #2 \\
 #3 & #4 \\
 \end{array} } \right)}

\numberwithin{equation}{section}

\begin{document}

\title[Galois representations ramified at one prime]{Galois representations ramified at one prime and with suitably large image}

\author[A.~Ray]{Anwesh Ray}
\address[Ray]{Centre de recherches mathématiques,
Université de Montréal,
Pavillon André-Aisenstadt,
2920 Chemin de la tour,
Montréal (Québec) H3T 1J4, Canada}
\email{anwesh.ray@umontreal.ca}

\maketitle

\begin{abstract}
Let $p\geq 7$ be a prime and $n>1$ be a natural number. We show that there exist infinitely many Galois representations $\varrho:\op{Gal}(\bar{\Q}/\Q)\rightarrow \op{GL}_{n}(\Z_p)$ which are unramified outside $\{p, \infty\}$ with large image. More precisely, the Galois representations constructed have image containing the kernel of the mod-$p^t$ reduction map $\op{SL}_n(\Z_p)\rightarrow \op{SL}_n(\Z/p^t\Z)$, where $t:=8(n^2-n)\left(3+\lfloor \op{log}_p(2^n+1)\rfloor\right)+8$. The results are proven via a purely Galois theoretic lifting construction. When $p\equiv 1\mod{4}$, our results are conditional since in this case, we assume a very weak version of Vandiver's conjecture. 
\end{abstract}

\section{Introduction}
\par Let $p$ be an odd prime number and $n>1$ be an integer. The inverse Galois problem posits that Galois extensions of $\Q$ (and other fields of interest) may be constructed so as to realize any finite group. Similarly, it is of natural interest for one to ponder over the existence of Galois representations with large image in a group of the form $\mathcal{G}(\Z_p)$, where $\mathcal{G}$ is a smooth reductive group-scheme over $\Z_p$. We consider the simplest case of natural interest, i.e., that of $\op{GL}_n$. The study of Galois representations is of central importance in number theory. The case when $n=2$ is special, since in this case, such representations abundantly arise from elliptic curves and modular forms. However, for larger values of $n$, it is natural to consider higher dimensional motives. For $n>2$, it proves difficult to construct geometric Galois representations whose image contains an open subgroup of $\op{SL}_n(\Z_p)$. For instance, a Galois representation that arises from the $p$-primary torsion in an abelian variety defined over $\Q$, must respect the Weil pairing, and this restricts the image.
\par Instead of relying on the study of motives and their associated Galois representations, we resort to purely Galois theoretic methods, and our constructions are based in Galois deformation theory. One hopes to construct large Galois representations with special ramification properties that do not necessarily arise from automorphic forms or motives in algebraic geometry. Let $E:=\Q(\mu_p)$ be the cyclotomic field extension of $\Q$ generated by the $p$-th roots of unity and $\widetilde{E}$ the maximal pro-$p$ extension of $E$ which is unramified at all finite primes $v$ of $E$ such that $v\nmid p$. A result of Shafarevich shows that when $p$ is a regular prime (i.e., $p$ does not divide the class number of $E$), the Galois group $\op{Gal}(\widetilde{E}/E)$ is a free pro-$p$ group with $\left(\frac{p+1}{2}\right)$-generators. R.~Greenberg uses this result to construct Galois representations $\rho:\op{Gal}(\bar{\Q}/\Q)\rightarrow \op{GL}_n(\Z_p)$ that are unramified at all finite primes outside $\{p\}$, whenever $p\geq 4\lfloor n/2\rfloor+1$ is a regular prime (cf. \cite[Proposition 1.1]{greenberg2016galois}). N.~Katz \cite{katz2020note} adopts a different approach, and constructs finitely ramified Galois representations with large image in $\op{GL}_n(\Z_p)$ for $p\equiv 1\mod{3}$ and $p\equiv 1\mod{4}$ and all $n\geq 6$. Such Galois representations are defined over certain cyclotomic fields $\Q(\mu_N)$ and are constructed via geometric techniques and a beautiful application of the Hilbert irreducibility theorem. Galois representations that are finitely ramified and with big image in more general reductive groups are constructed by S.~Tang (cf. \cite{tangalgebraic}). The above mentioned Galois representations are typically ramified at a large number of primes, especially when the dimension of the group $\mathcal{G}$ is large. Returning to the original result of Greenberg, consider the following refined question \cite[Question 1.1]{ray2021constructing}.

\begin{question}\label{question}
Let $p$ be any prime number and $n>1$ be an integer. Does there exist a continuous Galois representation $\rho:\op{Gal}(\bar{\Q}/\Q)\rightarrow \op{GL}_n(\Z_p)$ with suitably large image? If so, can one control the set of primes at which it may ramify? In particular, can such Galois representations be constructed so as to be unramified at all primes $\ell\notin \{p,\infty\}$?
\end{question}

By \emph{having suitably large image}, we mean that the image contains an open subgroup of $\op{SL}_n(\Z_p)$. There are parallels between the above question certain refinements of the inverse Galois problem. For instance, Abhyankar's conjecture posits the existence of finite Galois extensions of function fields in positive characteristic, with prescribed Galois group and a predetermined set of ramification subject to certain constraints. Such Galois extensions were constructed via the \emph{patching method} by D.~Harbater (cf. \cite{harbater1994abhyankar}). We also refer to \cite{harbater2018abhyankar} and references therein for further details. It does lead us to wonder if there are indeed constraints on the number of primes that may ramify in Galois representations to $\op{GL}_n(\Z_p)$ with suitably large image.
\par There has been some recent interest Question \ref{question}. Let $\Q_{\{p\}}$ be the maximal algebraic extension of $\Q$ which is unramified outside $\{p, \infty\}$, and set $\op{G}_{\{p\}}:=\op{Gal}\left(\Q_{\{p\}}/\Q\right)$. Let $\bar{\chi}$ denote the mod-$p$ cyclotomic character. We shall let $\mathcal{C}$ denote the mod-$p$ class group of the cyclotomic field $\Q(\mu_p)$. Recall that $p$ is regular if $\mathcal{C}=0$ and irregular otherwise. Note that $\mathcal{C}$ is a module over $\op{Gal}\left(\Q(\mu_p)/\Q\right)$, and therefore over $\Gp$. As a Galois module, $\cC$ decomposes into eigenspaces $\mathcal{C}=\bigoplus_{i=0}^{p-2} \cC(\bar{\chi}^i)$,
where $\cC(\bar{\chi}^i):=\{\mathcal{I} \in \cC\mid g \cdot \mathcal{I}=\bar{\chi}^i(g)\mathcal{I}\}$. Note that when $i$ is even, $\cC(\bar{\chi}^i)$ arises from the mod-$p$ class group of $\Q(\mu_p)^+$, the totally real subfield of $\Q(\mu_p)$. Vandiver's conjecture predicts that the mod-$p$ class group of $\Q(\mu_p)^+$ is equal to $0$, and thus, each even eigenspace of $\mathcal{C}$ is $0$. The \emph{index of irregularity} $e(p)$ is the number of integers $i$ that are in the range $0\leq i\leq p-2$ such that $\mathcal{C}(\bar{\chi}^i)\neq 0$. Given $k>0$, let $\mathcal{U}_k$ denote the kernel of the mod-$p^k$ reduction map 
\[\op{SL}_n(\Z_p)\rightarrow \op{SL}_n(\Z/p^k\Z).\]
In \cite{ray2021constructing}, the author shows that for $n>1$, $p\geq 2^{n+2+2e(p)}$, there are infinitely many Galois representations $\rho:\op{G}_{\{p\}}\rightarrow \op{GL}_n(\Z_p)$ whose image contains $\mathcal{U}_4$. This shows that Greenberg's result can be extended to irregular primes as well, provided the prime $p$ is suitably large. This result has been extended to more general groups by S.~Maletto \cite{maletto2022galois} and S.~Tang \cite{tang2022note}. Subsequently, C.~Maire \cite{maire2021galois} shows that for all odd primes $p$, there is a Galois representation $\rho:\op{Gal}(\bar{\Q}/\Q)\rightarrow \op{GL}_n(\Z_p)$ with big image, which is unramified outside $\{2, p, \infty\}$. 
\subsection{Statement of the main result} We now state the main result, which entirely resolves Question \ref{question} for all $n>1$ and all primes $p\geq 7$. The case $n=1$ is an easy exercise. When $p\equiv 1\mod{4}$, we need to assume a weak form of Vandiver's conjecture, which we briefly explain. 
\begin{ass*}[Assumption \ref{main ass}]\label{main ass intro}
There is an odd positive integer $k$ that lies in the range $3\leq k \leq \left(\frac{p-1}{2}\right)$ such that $\mathcal{C}(\bar{\chi}^{1+ k})=\mathcal{C}(\bar{\chi}^{p-k})=0$.
\end{ass*}
The above assumption is clearly a consequence of Vandiver's conjecture, which predicts that all even eigenspaces of $\mathcal{C}$ must vanish. In fact, it is easy to see that if the number of non-zero even eigenspaces of $\mathcal{C}$ is strictly less than $\lfloor \frac{p-2}{4}\rfloor$, then the above assumption is automatically satisfied (cf. Proposition \ref{choice of k}). Also, for primes $p\equiv 3\mod{4}$ such that $p\geq 7$, the above assumption is automatically satisfied (cf. Remark \ref{remark on 3 mod 4}). 

\begin{theorem*}[Theorem \ref{main thm}]
Let $p\geq 7$ be a prime and $n>1$ be a natural number. Suppose that the above assumption holds. Then, there exist infinitely many Galois representations $\varrho:\op{Gal}(\bar{\Q}/\Q)\rightarrow \op{GL}_{n}(\Z_p)$ which are unramified outside $\{p, \infty\}$ whose image contains $\mathcal{U}_t$, where $t:=8(n^2-n)\left(3+\lfloor \op{log}_p(2^n+1)\rfloor\right)+8$.
\end{theorem*}

\subsection{Method of proof} Let us briefly summarize the construction. We begin with a suitable choice of a residual represenation \[\bar{\rho}:\Gp\rightarrow \op{GL}_n(\Z/p\Z)\] which is carefully lifted to a characteristic zero representation via a step by step process. This residual representation is a \emph{diagonal representation}, i.e., a direct sum of $1$-dimensional characters, defined in section \ref{s 2.2}. Set $M:=(n^2-n)\left(3+\lfloor \op{log}_p(2^n+1)\rfloor\right)+1$ and $N:=2M$. A diagonal representation 
\[\rho_\beta=\left( {\begin{array}{ccccc}
   \beta_1 & & & & \\
    & \beta_2& & & \\
    & & \ddots & & \\
    & & & \beta_{n-1} & \\
    & & & & \beta_{n}
  \end{array} } \right):\Gp\rightarrow \op{GL}_n(\Z/p^N\Z),\] that lifts $\bar{\rho}$ is carefully chosen. We set $\rho_N:=\rho_\beta$ and set $\rho_M$ to denote the mod-$p^M$ reduction of $\rho_\beta$. Any other lift $\rho_N':\Gp\rightarrow \op{GL}_n(\Z/p^N\Z)$ of $\rho_M$ may be constructed by \emph{twisting} $\rho_N$ by a suitable cohomology class taking values in the adjoint representation $\op{Ad}\rho_M$ (cf. the discussion following Definition \ref{def 3.1}). We carefully construct a suitable cohomology class $f\in H^1(\Gp, \op{Ad}\rho_M)$ and then set $\rho_N'$ to denote the twist of $\rho_N$ by $f$ (cf. \eqref{def of rho N prime}). This twist has many advantageous properties. Let $L:=\Q(\rho_M)$ be the finite Galois extension of $\Q$ which is fixed by the kernel of $\rho_M$. One of the key properties exhibited by $\rho_N'$ is that $\rho_N'\left(\op{Gal}(\Q_{\{p\}}/L)\right)$ has suitably large image in 
  \[\mathcal{A}:=\op{ker}\left(\op{GL}_n(\Z/p^N\Z)\rightarrow \op{GL}_n(\Z/p^M\Z)\right).\] It is easy to see that $\mathcal{A}$ is an abelian group and can be identified with $n\times n$-matrices with coefficients in $\Z/p^M\Z$. By virtue of being an abelian group, $\mathcal{A}$ acquires a natural Galois action and is isomorphic to the adjoint representation $\op{Ad}\rho_M$. Of crucial importance is that this adjoint representation decomposes into a direct sum of a trivial module arising from the diagonal torus, and a direct sum of distinct cyclic Galois-modules. It should be noted that the residual adjoint representation $\op{Ad}\bar{\rho}$ does not necessarily have this favorable property. The infinitesimal deformation problem associated to $\bar{\rho}$ is unobstructed, i.e., $H^2(\Gp, \op{Ad}\bar{\rho})=0$ (cf. Proposition \ref{prop unobstructed}). As a result, it follows (cf. Proposition \ref{prop 4.1}) that $\rho_N'$ exhibits infinitely many lifts to characteristic zero. The twist $\rho_N'$ satisfies a number of key properties, and Proposition \ref{main prop} implies that the lifts of $\rho_N'$ all have image containing a large open subgroup of $\op{SL}_n(\Z_p)$.
\subsection{Organization} Including the introduction, the manuscript consists of five sections. In section \ref{s 2}, we introduce notation and prove a number of preliminary results. In this section, we recall the main assumption when $p\equiv 1\mod{4}$ (i.e., Assumption \ref{main ass}) and explain why it is significantly weaker than Vandiver's conjecture. We then introduce the choice of residual representation $\bar{\rho}$, and show that it is unobstructed (cf. Proposition \ref{prop unobstructed}). Also in this section, we discuss properties of diagonal representations and introduce the notion of \emph{admissibility}. In section \ref{s 3}, we systematically introduce the deformation theory of Galois representations, and discuss how deformations of a given Galois representation may be \emph{twisted} by cohomology classes valued in the adjoint module. In section \ref{s 4}, we prove results that allow us to state conditions for the image of a Galois representation to be suitably large in $\op{GL}_n(\Z_p)$. The main criterion proven in this section is Proposition \ref{main prop}, which is crucially used in the proof of the main theorem. Finally, in section \ref{s 5}, we give a proof of the main result of the manuscript.

  \subsection{Acknowledgment} The author is supported by the CRM-Simons postdoctoral fellowship. He thanks the anonymous referee for the helpful report.
\section{Preliminaries}\label{s 2}
\par This section is preliminary in nature. We introduce some key notation that will aid our constructions in later sections, and prove a few basic yet crucial results. Throughout $p$ will be a fixed prime number and $n>1$ an integer. We assume throughout that $p\geq 7$. Set $\Q_p$ to denote the field of $p$-adic numbers and $\Z_p:=\varprojlim_m \Z/p^i\Z$ the ring of $p$-adic integers. Let $\bar{\Q}$ (resp. $\bar{\Q}_p$) denote the algebraic closure of $\Q$ (resp. $\Q_p$). We let $\F_p$ denote the field with $p$ elements, i.e., $\F_p:=\Z/p\Z$. At each prime $\ell$, let $\iota_\ell: \bar{\Q}\hookrightarrow \bar{\Q}_\ell$ be a chosen embedding, and set $\op{G}_\ell$ to denote the absolute Galois group $\op{Gal}(\bar{\Q}_\ell/\Q_\ell)$. The choice of $\iota_\ell$ corresponds to a choice of a prime that lies above $\ell$, and an inclusion of $\op{G}_\ell$ into the absolute Galois group $\op{G}_{\Q}:=\op{Gal}(\bar{\Q}/\Q)$. Let $\op{I}_\ell$ be the inertia subgroup of $\op{G}_\ell$. Given a unital commutative ring $R$, let $M_n(R)$ consist of all $n\times n$ matrices with entries in $R$. Let $e_{i,j}(R)$ denote the matrix in $M_n(R)$ with $0$s in all positions except at $(i,j)$, where the entry is equal to $1$. When there is no cause for confusion, we simply write $e_{i,j}$ in place of $e_{i,j}(R)$.

\par It benefits us to consider quotients of the absolute Galois group $\op{G}_{\Q}$ in which all primes outside a prescribed finite set are unramified. Given a finite set of prime numbers $S$, denote by $\Q_S$ the maximal extension of $\Q$ which is contained in $\bar{\Q}$ in which all primes $\ell\notin S\cup \{\infty\}$ are unramified. Note that the extension $\Q_S$ is necessarily Galois over $\Q$. Set $\op{G}_S$ to denote the Galois group $\op{Gal}(\Q_S/\Q)$. Given a Galois representation $\rho: \op{G}_S\rightarrow \op{GL}_n(\cdot)$, let $\rho_{|\ell}$ denote the restriction of $\rho$ to $\op{G}_\ell$. We say that $\rho$ is \emph{unramified at $\ell$} if $\rho_{|\ell}$ is trivial when restricted to $\op{I}_\ell$. For each integer $m\geq 1$ let $\mu_{p^m}$ denote the $p^m$-th roots of unity. Let $\chi:\op{G}_{\{p\}}\rightarrow \Z_p^\times$ be the $p$-adic cyclotomic character and denote by $\bar{\chi}$ (resp. $\chi_m$) the mod-$p$ (resp. mod-$p^m$) reduction of $\chi$. By abuse of notation, when it is clear from the context, we shall simply write $\chi$ instead of $\chi_m$. 

\begin{proposition}\label{choice of k} Suppose that the number of non-zero even eigenspaces of $\mathcal{C}$ is strictly less than $\lfloor \frac{p-2}{4}\rfloor$. Then, there is an odd positive integer $k$ that lies in the range $3\leq k \leq \left(\frac{p-1}{2}\right)$ such that $\mathcal{C}(\bar{\chi}^{1+ k})=\mathcal{C}(\bar{\chi}^{p-k})=0$.
\end{proposition}
\begin{proof}
Suppose that by way of contradiction that for each odd $k$ in the range $3\leq k \leq \left(\frac{p-1}{2}\right)$, either $\mathcal{C}(\bar{\chi}^{1+ k})$ or $\mathcal{C}(\bar{\chi}^{p-k})$ is non-zero. First, suppose that $p\equiv 1\mod{4}$, i.e., $p=4s+1$. Then, the number of odd values of $k$ in the range $3\leq k \leq \left(\frac{p-1}{2}\right)=2s$ is $(s-1)=\lfloor \frac{p-2}{4}\rfloor$. On the other hand, suppose that $p\equiv 3\mod{4}$, i.e., $p=4s+3$. Then, the number of odd values of $k$ is $s=\lfloor \frac{p-2}{4}\rfloor$. In both cases, the number of even eigenspaces must be $\geq \lfloor \frac{p-2}{4}\rfloor$, a contradiction. Therefore, there must exist $k$ as above such that $\mathcal{C}(\bar{\chi}^{1+ k})=\mathcal{C}(\bar{\chi}^{p-k})=0$.
\end{proof}

We make the above into an assumption, which is clearly significantly weaker than Vandiver's conjecture.

\begin{ass}\label{main ass}
There is an odd positive integer $k$ that lies in the range $3\leq k \leq \left(\frac{p-1}{2}\right)$ such that $\mathcal{C}(\bar{\chi}^{1+ k})=\mathcal{C}(\bar{\chi}^{p-k})=0$.
\end{ass}

\begin{remark}\label{remark on 3 mod 4}
When $p\equiv 3\mod{4}$ and $k=\left(\frac{p-1}{2}\right)$, then, the above assumption stipulates that $\mathcal{C}(\bar{\chi}^{\frac{p+1}{2}})=0$. This result is proven by Osburn (see \cite{osburn}). Therefore, the above is an assumption only for $p\equiv 1\mod{4}$.
\end{remark} 
We fix an integer $k$ in the range $3\leq k \leq \left(\frac{p-1}{2}\right)$ such that the above assumption holds and let $\nu$ denote the character $\bar{\chi}^k$. 

\subsection{Diagonal Galois representations}\label{s 2.2}
Our construction is based on the study of suitable lifts of diagonal Galois representations, i.e., Galois representations that are expressed as a direct sum of $1$-dimensional characters. Throughout, we shall fix a suitable residual representation. Let $\bar{\rho} :\Gp\rightarrow \op{GL}_n(\F_p)$ be the diagonal Galois representation given by \[\bar{\rho}=\bar{\rho}^{(n)}=\left( {\begin{array}{ccccc}
   \alpha_1 & & & & \\
    & \alpha_2& & & \\
    & & \ddots & & \\
    & & & \alpha_{n-1} & \\
    & & & & \alpha_{n}
  \end{array} } \right),\] where
  \begin{equation}\label{defn of alpha}
  \alpha_i=\begin{cases}& 1\text{ if }i \text{ is odd,}\\
& \nu \text{ if }i \text{ is even.}\\
\end{cases}\end{equation}

Thus the representations we consider for $n=2, 3, 4,\dots$ are as follows
\[\bar{\rho}^{(2)}=\mtx{1}{}{}{\nu}, \bar{\rho}^{(3)}=\left( {\begin{array}{ccc}
   1 & &  \\
    & \nu &  \\
    & & 1
  \end{array} } \right) \text{ and }\bar{\rho}^{(4)}=\left( {\begin{array}{cccc}
   1 & & & \\
    & \nu & &  \\
    & & 1 &  \\
    & & & \nu 
  \end{array} } \right), \dots, \text{etc.} \]
Henceforth, we suppress the dependence of $n$ in our notation, and simply denote the residual representation by $\bar{\rho}$.
We set $\alpha$ to denote the tuple $(\alpha_1, \dots, \alpha_n)$, and given an integer $N>0$, let $\beta=(\beta_1, \dots, \beta_n)$ be a tuple of characters $\beta_i:\Gp\rightarrow \left(\Z/p^N \Z\right)^\times$. We say that $\beta_i$ \emph{lifts} $\alpha_i$ if $\alpha_i=\beta_i\mod{p}$, and that $\beta$ lifts $\alpha$ if $\beta_i$ lifts $\alpha_i$ for all $i=1, \dots, n$. In order to emphasize the dependence on $N$, we say that that $\beta$ is an \emph{$N$-lift} of $\alpha$. 
\begin{definition}\label{admissible defn}
Let $m$ and $N$ be positive integers such that $m\leq N$. Let $\beta=(\beta_1, \dots, \beta_n)$ be an $N$-lift of $\alpha$. We say that the tuple $\beta$ is $m$-\emph{admissible} if the following conditions are satisfied.
\begin{enumerate}
\item\label{c1 of admissible defn} For $i=1, \dots, n$, the character $\beta_i$ is of the form $\chi^{k_i}$ for some integer $k_i\geq 1$.
\item\label{c2 of admissible defn} The characters $\beta_i$ are all distinct modulo $p^m$. 
\item\label{c3 of admissible defn} As $(i,j)$ ranges over all pairs such that $i\neq j$, the characters $\beta_{i,j}:=\beta_i\beta_j^{-1}$ are all distinct modulo $p^m$.
\end{enumerate}
\end{definition}
Thus condition \eqref{c2 of admissible defn} implies that $\beta_{i,j}\neq 1$ modulo $p^m$ whenever $i\neq j$.
\begin{remark}
We observe the following.
\begin{enumerate}
\item Since $\beta$ lifts $\alpha$, we find that $k_i\equiv \left(\frac{1+(-1)^i}{2}\right) k\mod{(p-1)}$ for all $i$.
\item Note that \eqref{c2 of admissible defn} is equivalent to the condition that $k_i\not \equiv k_j \mod{p^{m-1}(p-1)}$ for $i\neq j$. By the above remark, this condition is automatically satisfied if $i\not \equiv j\mod{2}$.
\item The condition \eqref{c3 of admissible defn} is equivalent to the condition that $k_i-k_j\not\equiv k_s-k_t\mod{p^{m-1}(p-1)}$ for all distinct pairs $(i,j)$ and $(s,t)$ such that $i\neq j$ and $s\neq t$.
\end{enumerate}
\end{remark}

\par Given an $N$-lift $\beta$ of $\alpha$, we consider the associated Galois representation 
\[\rho_\beta=\left( {\begin{array}{ccccc}
   \beta_1 & & & & \\
    & \beta_2& & & \\
    & & \ddots & & \\
    & & & \beta_{n-1} & \\
    & & & & \beta_{n}
  \end{array} } \right):\Gp\rightarrow \op{GL}_n(\Z/p^N\Z).\] Note that since $\beta$ is a lift of $\alpha$, it follows that $\rho_\beta$ is a lift of $\bar{\rho}$. We say that $\rho_\beta$ is $m$-admissible if $\beta$ is $m$-admissible.
  \par The following result shows that for any integer $n>1$, an $m$-adimissible lift of $\alpha$ does exist for an explicit choice of $m$.

  \begin{lemma}\label{choice of m}
  Let $m$ be an integer such that $m>2+\op{log}_p(2^n+1)$, and let $N\geq m$. For an integer $i$ in the range $1\leq i\leq n$, let $\epsilon_i:=\frac{1+(-1)^i}{2}k$, $k_i:=2^i p + \epsilon_i$, and set $\beta_i:=\chi^{k_i}\mod{p^N}$. Then, $\beta=(\beta_1, \dots, \beta_n)$ is an $m$-admissible $N$-lift of $\alpha$. 
  \end{lemma}
  
\begin{proof}
Note that since $p^{m-1}>p(2^n+1)$, it follows that $\frac{1}{2}p^{m-1}(p-1)>k_n=p2^n+\epsilon_n$. Since $k_i\neq k_j$, it follows that $k_i\not \equiv k_j\mod{p^{m-1}(p-1)}$. On the other hand, suppose there are two distinct pairs $(i,j)$ and $(s,t)$ such that $i\neq j$, $s\neq t$, and \[k_i-k_j\equiv k_s-k_t\mod{p^{m-1}(p-1)}.\] Then, since 
\[\op{max}\{|k_i-k_j|, |k_s-k_t|\}<\frac{1}{2}p^{m-1}(p-1),\]we deduce that $k_i-k_j=k_s-k_t$. Without loss of generality, assume that $i>j,s, t$. Then, we find that $p(2^i-2^j)=p(2^s-2^t)+b$, where $|b|
<p$. Then, we find that $2^i-2^j=2^s-2^t$ and $b=0$. But then, $2^i=2^j+2^s-2^t<2\cdot 2^{i-1}=2^i$, a contradiction. Hence, we have proven that $\beta$ is $m$-admissible.
\end{proof}

\subsection{The adjoint representation} Let $\g=M_n(\F_p)$ denote the Lie algebra of $\op{GL}_n(\F_p)$, consisting of $n\times n$ matrices with entries in $\F_p=\Z/p\Z$. Denote by $\g_0$ the subspace of $\g$ consisting of matrices with trace equal to $0$. We give $\g$ the structure of a Galois module as follows. Let $v\in \g$ and $g\in \Gp$. Then, $g\cdot v$ is defined to be $\bar{\rho}(g) v \bar{\rho}(g)^{-1}$. The action is referred to as the adjoint action. When referring to the Galois module, we simply use $\op{Ad}\bar{\rho}$, and when referring to the underlying space, we use $\g$. Note that $\g_0$ which consists of trace zero matrices is a submodule, which we also denote by $\op{Ad}^0\bar{\rho}$. Set $\alpha_{i,j}$ to denote $\alpha_i \alpha_j^{-1}$. Recall that $\alpha_i=\bar{\chi}^{\left(\frac{1+(-1)^i}{2}\right)k}$, and therefore, we find that 
\[\alpha_{i,j}=\begin{cases}
& 1 \text{ if }i\equiv j\mod{2},\\
& \nu=\bar{\chi}^k \text{ if i is even and j is odd},\\
& \nu^{-1}=\bar{\chi}^{-k}\text{ if i is odd and j is even.}
\end{cases}\]
Let $\mathfrak{t}$ (resp. $\mathfrak{t}_0$) denote the submodule of $\op{Ad}\bar{\rho}$ (resp. $\op{Ad}^0\bar{\rho}$) consisting of diagonal matrices. Note that the Galois action on $\mathfrak{t}$ is trivial and we have the following decompositions
\[\begin{split} & \op{Ad}\bar{\rho}=\mathfrak{t}\oplus \left(\bigoplus_{i\neq j} \F_p(\alpha_{i,j})\right), \\
& \op{Ad}^0\bar{\rho}=\mathfrak{t}_0\oplus \left(\bigoplus_{i\neq j} \F_p(\alpha_{i,j})\right). 
    \end{split}\]
Given a module $M$ over $\op{G}_S$ and an integer $i\geq 0$, we denote by $H^i(\op{G}_S, M)$ the $i$-th cohomology group. Denote by $\Sh^i_S(M)$ the kernel of the restriction maps
\[\Sh^i_S(M):=\op{ker}\left\{H^i(\op{G}_S, M)\rightarrow \bigoplus_{\ell\in S} H^i(\op{G}_\ell, M)\right\}. \]
\par Given a representation $\bar{\rho}:\op{G}_S\rightarrow \op{GL}_n(\F_p)$, we say that $\bar{\rho}$ is \emph{$S$-unobstructed} if the cohomology group $H^2(\op{G}_S, \op{Ad}\bar{\rho})$ vanishes. When $\bar{\rho}$ is unramified at all primes $\ell\neq p$, we shall say that $\bar{\rho}$ is unobstructed to mean that it is $\{p\}$-unobstructed. When $\bar{\rho}$ is the residual representation $\op{diag}\left(\alpha_1, \dots, \alpha_n\right)$, the representation $\bar{\rho}$ is always unobstructed, as the following result shows.

\begin{proposition}\label{prop unobstructed}
Let $n>1$ be an integer, and let $\bar{\rho}$ be the Galois representation $\op{diag}\left(\alpha_1, \dots, \alpha_n\right)$, where $\alpha_i$ is given by \eqref{defn of alpha}. Then, we find that $H^2(\Gp, \op{Ad}\bar{\rho})=0$.
\end{proposition}
\begin{proof}
Consider the exact sequence 
\[0\rightarrow \Sh_{\{p\}}^2(\op{Ad}\bar{\rho})\rightarrow H^2(\Gp, \op{Ad}\bar{\rho})\rightarrow H^2(\op{G}_p, \op{Ad}\bar{\rho}).\]It suffices to show that $\Sh_{\{p\}}^2(\op{Ad}\bar{\rho})=0$ and that $H^2(\op{G}_p, \op{Ad}\bar{\rho})=0$. Note that the Galois module $\op{Ad}\bar{\rho}$ decomposes as follows
\[\op{Ad}\bar{\rho}=\mathfrak{t}\oplus\left(\ \bigoplus_{i\neq j} \F_p(\alpha_{i,j})\right).\]Note that the Galois action on $\mathfrak{t}$ is the trivial one. Since the characters in the above decomposition all belong to $\{1, \nu, \nu^{-1}\}$, we are to show that $\Sh^2_{\{p\}}\left(\F_p(\psi)\right)=0$ and that $H^2(\Gp, \F_p(\psi))=0$ for all characters $\psi\in \{1, \nu, \nu^{-1}\}$.

\par First let us show that $\Sh^2_{\{p\}}\left(\F_p(\psi)\right)=0$ for the above mentioned characters. Note that by global duality of $\Sh$-groups \cite{neukirch2013cohomology}, we have that $\Sh^2_{\{p\}}\left(\F_p(\psi)\right)$ are $\Sh^1_{\{p\}}\left(\F_p(\bar{\chi}\psi^{-1})\right)$ are naturally dual to one another. Therefore, we are to check to that \[\Sh^1_{\{p\}}\left(\F_p(\bar{\chi}\psi^{-1})\right)=0\] for the characters $\psi\in \{1, \nu, \nu^{-1}\}$. A straightforward application of inflation-restriction shows that $\Sh^1_{\{p\}}\left(\F_p(\bar{\chi}\psi^{-1})\right)=0$ if $\mathcal{C}(\bar{\chi}\psi^{-1})=0$ (cf. \cite[Lemma 3.2]{ray2021constructing}). It is a well known fact that the first eigenspace $\mathcal{C}(\bar{\chi})=0$, cf. \cite[Proposition 6.16]{washington1997introduction}. On the other hand, the vanishing of $\mathcal{C}(\bar{\chi}\nu^{\pm 1})$ follows from Assumption \ref{main ass}. Therefore, we deduce that $\Sh^2_{\{p\}}(\op{Ad}\bar{\rho})=0$. 
\par Next, we show that $H^2(\op{G}_p, \op{Ad}\bar{\rho})=0$. It suffices to show that $H^2\left(\op{G}_p, \F_p(\psi)\right)=0$ for each character $\psi\in \{1, \nu, \nu^{-1}\}$. By local duality, we have that $H^2\left(\op{G}_p, \F_p(\psi)\right)$ is dual to $H^0\left(\op{G}_p, \F_p(\bar{\chi}\psi^{-1})\right)$. Since $\psi_{|\op{G}_p}$ is not equal to $\bar{\chi}$, it follows that $H^0\left(\op{G}_p, \F_p(\bar{\chi}\psi^{-1})\right)=0$, and thus, $H^2\left(\op{G}_p, \F_p(\psi)\right)=0$ for $\psi$ as above. It follows that $H^2(\op{G}_p, \op{Ad}\bar{\rho})=0$ and it follows that $H^2(\op{G}_{\{p\}}, \op{Ad}\bar{\rho})=0$.
\end{proof} 
\section{Deformation theory}\label{s 3}
\par In this section, we study the deformation theory of diagonal representations. The results in this section shall further aid our lifting construction. For a comprehensive treatment of the subject, the reader may refer to \cite{mazur1997introduction}.

\par Let $G$ be a group, $R$ be a commutative unital local ring, with maximal ideal $\mathfrak{m}_R$. Recall that $M_n(R)$ is the ring of $n\times n$-matrices with entries in $R$. Given a representation $\varrho:G\rightarrow \op{GL}_n(R)$, the associated \emph{adjoint module} denoted $\op{Ad}\varrho$, is the module $M_n(R)$ on which $g\in G$ acts via $g\cdot v:=\varrho(g) v\varrho(g)^{-1}$. There is a direct sum decomposition
\[\op{Ad}\varrho=\op{Ad}^0\varrho\oplus \op{R}\cdot \op{I}_n,\]
where $\op{Ad}^0(R)$ consists of matrices with trace equal to $0$ and $\op{I}_n$ is the identity in $M_n(R)$. We note that the Galois action on the second factor in the above decomposition is trivial.

\begin{definition}
    Let $\pi:S\rightarrow R$ be a homomorphism of commutative unital local rings and let $I$ denote the kernel of $\pi$. In this manuscript, we shall say that $\pi$ is a \emph{small-extension} if the following conditons hold
\begin{enumerate}
\item $\pi$ is surjective, 
\item $I^2=0$,
\item $I$ is principal.
\end{enumerate}
\end{definition}
Given a small extension $\pi$, let $\pi^*:\op{GL}_n(S)\rightarrow \op{GL}_n(R)$ be the homomorphism obtained by applying $\pi^*$ to each entry. 
\par The following example illustrates the above definition. Note that if $M$ and $N$ are integers such that $M<N$, then, the quotient map 
\[\pi_{M, N}: \Z/p^{N}\Z\rightarrow \Z/p^{M}\Z\] is a small extension provided $N\leq 2M$ and that the kernel of $\pi_{M,N}$ is generated by $p^M$.
\par Let us now introduce some further notation. The ideal $I$ is generated by the element $t$. Consider the map $R\rightarrow I$ sending $r$ to $tr$. Let $J$ denote the kernel of this map and set $R'$ to denote the quotient $R/J$. Thus the map sending $r$ to $tr$ induces an isomorphism of $R$-modules $R'\xrightarrow{\sim} I$. Let us introduce a little bit of notation which will be used throughout the manuscript. Given a matrix $A\in M_n(R')$ let $\widetilde{A}$ be a lift of $A$ to $M_n(S)$. Let \[t\widetilde{A}\in \op{ker}\left\{M_n(S)\rightarrow M_n(R)\right\}\] be the matrix obtained by multiplying all entries of $\widetilde{A}$ by $t$. Clearly, the matrix $t\widetilde{A}$ is independent of the choice of lift $\widetilde{A}$, and \emph{we denote this product by $tA$}.
\begin{proposition}\label{prop 3.1}
Let $\pi$ be a small extension. Then, the kernel of $\pi^*:\op{GL}_n(S)\rightarrow \op{GL}_n(R)$ defined above is abelian and can be identified with $M_n(R')$. The association takes a matrix $A\in M_n(R')$ to $\op{I}_n+tA$ in $\op{ker}\left(\op{GL}_n(S)\rightarrow \op{GL}_n(R)\right)$. 
\end{proposition}
\begin{proof}
Since $t^2=0$, it follows that $\op{I}_n+tA$ is an invertible matrix with inverse $\op{I}_n-tA$. Moreover, the map sending $A\mapsto \op{I}_n+tA$ is seen to be a homomorphism, due to the relation 
\[
\begin{split}&\left(\op{I}_n+tA\right)\left(\op{I}_n+tB\right)\\
=&\left(\op{I}_n+t\widetilde{A}\right)\left(\op{I}_n+t\widetilde{B}\right)\\
=& \op{I}_n+t(\widetilde{A}+\widetilde{B})\\
=& \op{I}_n+t(A+B).\end{split}\] It is also clear that any element in the kernel of $\pi^*$ is of the form $\op{I}_n+tA$, where $A\in M_n(R')$. 
\end{proof}

\par Our discussion returns to the map $\pi_{M, N}$ (subject to the condition that $N\leq 2M$). We set $k:=N-M$ and we find that $R'=\Z/p^k\Z$. The Proposition \ref{prop 3.1} asserts that the kernel of the reduction map $\op{GL}_n(\Z/p^{N}\Z)\rightarrow \op{GL}_n(\Z/p^M\Z)$ is identified with $M_n(\Z/p^k\Z)$, upon associating a matrix $A\in M_n(\Z/p^k\Z)$ with $\op{I}_n+p^MA$. For instance, say $n=2$, $M=2$, $N=4$ and $k=2$. Consider the matrix $A=\mtx{0}{1+p}{-1}{0}$ and let $t=p^2$. Then, the product $p^2A$ is given by $\mtx{0}{p^2+p^3}{-p^2}{0}$, and $\op{I}_2+p^2\mtx{0}{1+p}{-1}{0}$ is given by $\mtx{1}{p^2+p^3}{-p^2}{1}\in M_2(\Z/p^4 \Z)$.

\par Let us now discuss the notion of a \emph{deformation}. Let $\pi: S\rightarrow R$ be any surjective map of commutative rings and $G$ be a group. Let $\rho:G\rightarrow \op{GL}_n(R)$ be a representation of $G$ on a free $R$-module of rank $n$. 

\begin{definition}\label{def 3.1}
A \emph{lift} of $\rho$ is a representation $\widetilde{\rho}:G\rightarrow \op{GL}_n(S)$, which fits into a commutative triangle \[ \begin{tikzpicture}[node distance = 2.5 cm, auto]
            \node at (0,0) (G) {$G$};
             \node (A) at (3,0){$\op{GL}_n(R)$.};
             \node (B) at (3,2){$\op{GL}_n(S)$};
      \draw[->] (G) to node [swap]{$\rho$} (A);
       \draw[->] (B) to node{} (A);
      \draw[->] (G) to node {$\tilde{\rho}$} (B);
\end{tikzpicture},\] where the vertical map is $\pi^*$. Two lifts $\tilde{\rho}$ and $\tilde{\rho}'$ are said to be \emph{strictly equivalent} if $\widetilde{\rho}'=B \widetilde{\rho} B^{-1}$ for some matrix $B$ in $\op{ker}\left(\op{GL}_n(S)\xrightarrow{\pi^*} \op{GL}_n(R)\right)$. A \emph{deformation} of $\rho$ is a \emph{strict equivalence class} of lifts. 
\end{definition}
For the rest of this discussion, let us assume that $\pi:S\rightarrow R$ is a small extension, and that $I\simeq R'\simeq R/J$ as an $R$-module.
Given $\rho:G\rightarrow \op{GL}_n(R)$, let us assume that there exists a deformation $\widetilde{\rho}:G\rightarrow \op{GL}_n(S)$. Then, any other deformation is given by $r=(\op{I}_n+t f)\tilde{\rho}$ for a uniquely determined cohomology class $f\in H^1(G, \op{Ad}\rho_0)$, where $\rho_0$ is the reduction of $\rho$ modulo $J$. In greater detail, let $F$ be a cocycle representing $f$, consider the representation given by the formula $(\op{I}_n+t F)\widetilde{\rho}:G\rightarrow \op{GL}_n(S)$. This prescribes a well-defined lift of $\rho$. The cocyle relation for $F$ implies that the following relation holds for all $h_1, h_2\in G$, 
\[\left(\op{I}_n+t F(h_1)\right)\tilde{\rho}(h_1)\left(\op{I}_n+t F(h_2)\right)\tilde{\rho}(h_2)=\left(\op{I}_n+t F(h_1 h_2)\right)\tilde{\rho}(h_1 h_2).\] Thus, $(\op{I}_n+tF)\widetilde{\rho}$ defines a well defined homomorphism, and $(\op{I}_n+tf)\widetilde{\rho}$ defines a well defined deformation of $\rho$. Indeed, if $F_1$ and $F_2$ are two cocycles that differ by a coboundary $z=F_2-F_1$, we have that 
\begin{equation}\label{boring eq 1}(\op{I}_n+t F_2)\widetilde{\rho}=(\op{I}_n+t F_1)(\op{I}_n+t z)\widetilde{\rho}.\end{equation} By virtue of being a coboundary, \[z(h)=x-hx=x-\rho_0(h)x \rho_0(h)^{-1}.\]
Therefore, we find that 
\begin{equation}\label{boring eq 2}\begin{split}\left(\op{I}_n+t z(h)\right)= & \left(\op{I}_n+t x\right)\left(\op{I}_n+t \rho_0(h)x \rho_0(h)^{-1}\right)^{-1}\\
& \left(\op{I}_n+t x\right)\tilde{\rho}(h)\left(\op{I}_n+t x\right)^{-1}\tilde{\rho}(h)^{-1}.\\
\end{split}\end{equation}
Therefore, combining \eqref{boring eq 1} and \eqref{boring eq 2}, we find that 
\begin{equation}(\op{I}_n+tF_2)\tilde{\rho} =(\op{I}_n+tx)(\op{I}_n+tF_1)\tilde{\rho} (\op{I}_n+tx)^{-1},\end{equation} i.e., the lifts $(\op{I}_n+tF_1)\tilde{\rho}$ and $(\op{I}_n+tF_2)\tilde{\rho}$ are strictly equivalent. One thus checks that $(\op{I}_n+tf)\tilde{\rho}$ is a well defined deformation of $\rho$.
\par Recall that Lemma \ref{choice of m} guarantees the existence of an $m$-admissible lift of $\alpha$, for $m:=m(p,n)=3+\lfloor\op{log}_p(2^n+1)\rfloor$. It is perhaps of interest to note that $m$ decreases when $p$ gets larger, and hence for large enough values of $p$ and fixed $n$ the value of $m$ is $3$. We shall of course fix a value of both $n$ and $p$ in our discussion throughout the paper. Set $M$ to denote $m(n^2-n)+1$ throughout. Set $\A$ to denote the kernel of the mod-$p^M$ reduction map 
\[\A:=\op{ker}\left(\op{GL}_{n}(\Z/p^{2M})\rightarrow \op{GL}_{n}(\Z/p^{M})\right).\]
We set $N:=2M$ and set $\rho_N$ to denote the Galois representation $\rho_\beta$ where $\beta$ is an $m$-admissible $N$-lift of $\alpha$. Let $i$ be an integer which lies in the range $1\leq i\leq N$ and $\rho_i$ be the reduction of $\rho_N$ modulo $p^i$. Let $\gamma=(\gamma_1, \dots, \gamma_n)$ be the reduction of $\beta$ modulo $p^M$. Thus $\rho_M$ is the diagonal representation $\rho_{\gamma}$. 

\par Given a Galois representation $\varrho:\op{G}_{\Q}\rightarrow \op{GL}_n(R)$, the field $\Q(\varrho):=\bar{\Q}^{\op{ker}\varrho}$ is the field fixed by the kernel of $\varrho$. We shall refer to $\Q(\varrho)$ as the field \emph{cut out by $\varrho$} and we note that $\varrho$ induces an injection from $\op{Gal}(\Q(\varrho)/\Q)$ into $\op{GL}_n(R)$. If the ring $R$ is finite, then, $\Q(\varrho)$ is a finite Galois extension of $\Q$. From here on in, set $L:=\Q(\rho_M)$ to be the extension which is cut out by $\rho_M$. Let $\mathcal{M}$ be the submodule of $\mathcal{A}$ which is the image of $\rho_{N|\op{G}_L}:\op{G}_{L}\rightarrow \A$. Note that by Proposition \ref{prop 3.1}, $\A$ is identified with $\op{Ad}\rho_M$ upon identifying $\op{I}_n +p^M A$ with $A\in \op{Ad}\rho_M$. Thus $\mathcal{M}$ is viewed as a $\Z/p^M\Z$-submodule of $\op{Ad}\rho_M$, where 
\[\op{Ad}\rho_M=\mathfrak{t}_M\oplus \left(\bigoplus_{i\neq j} \Z/p^M\Z (\gamma_{i,j})\right),\]
where $\gamma_{i,j}:=\gamma_i\gamma_j^{-1}$.  Here, $\mathfrak{t}_M$ is the diagonal torus, with trivial action. 

\begin{proposition}\label{mathcal M prop}
With respect to notation above, $\mathcal{M}$ is a $\Gp$-stable submodule of $\op{Ad}\rho_M$.
\end{proposition}
\begin{proof}
Let $x\in \rho_N(\op{G}_L)$, and let $A\in \op{Ad}\rho_M$ be such that $x=\op{I}_n+p^MA$. Thus, in particular, $A$ is contained in $\mathcal{M}$, in accordance with the identification of $\mathcal{A}$ with $\op{Ad}\rho_M$, as is explained above. Let $h\in \op{G}_L$ be such that $x=\rho_N(h)$. For $g\in \Gp$, note that $ghg^{-1}$ is in $\op{G}_L$. We find that 
\[\begin{split}&\rho_N(ghg^{-1})=\rho_N(g)x\rho_N(g)^{-1}\\
=&\rho_N(g)\left(\op{I}_n+p^M A\right)\rho_N(g)^{-1}\\
=&\left(\op{I}_n+p^M \rho_M(g) A \rho_M(g)^{-1}\right)\\
=& \left(\op{I}_n+p^M g\cdot A \right).
\end{split}\]
Therefore, $g\cdot A$ is contained in $\mathcal{M}$.
\end{proof}
\par Recall that $e_{i,j}$ be the matrix with all entries $(s,t)\neq (i,j)$ equal to $0$ and $(i, j)$-th entry equal to $1$. For $i\neq j$, $e_{i,j}$ a generator of the submodule of $\op{Ad}\rho_M$ isomorphic to $\Z/p^M\Z(\gamma_{i,j})$. Given $x\in \A$, let $x_0$ be the projection of $x$ to $\mathfrak{t}_M$, and let $x_{i,j}$ be the projection to the $\gamma_{i,j}$-component. We find that $x_{i,j}=y_{i,j} e_{i,j}$, where $y_{i,j}\in \Z/p^M\Z$. Thus, we find that $x=x_0+\sum_{i,j} y_{i,j} e_{i,j}$, where the sum ranges over all pairs $(i,j)$ such that $i\neq j$. We remind the reader that $M:=m(n^2-n)+1$, $N:=2M$, and $\rho_N=\rho_\beta$, where $\beta$ is an $m$-admissible $N$-lift of $\alpha$.

\begin{lemma}\label{lemma 3.3}
With respect to notation above, let $(i,j)$ be a pair such that $1\leq i,j\leq n$ and $i\neq j$. Let $x\in \mathcal{M}$ and write $x_{i,j}=y_{i,j}e_{i,j}$, where $y_{i,j}\in\Z/p^M\Z$. 
\begin{enumerate}
    \item\label{p1 of lemma 3.3} Suppose that $y_{i,j}$ is not divisible by $p$, then, $\mathcal{M}$ contains $p^{m (n^2-n)} e_{i, j}$.
    \item\label{p2 of lemma 3.3}  Furthermore, $\mathcal{M}$ contains $p^{m(n^2-n)}x_0$.
\end{enumerate}
\end{lemma}

\begin{proof}
Let $(i,j)$ be a pair such that $i\neq j$ and assume that the $y_{i,j}$ is not divisible by $p$. Since $\gamma$ is $m$-admissible, $\gamma_{(s,t)}\not \equiv \gamma_{(i,j)}\mod{p^m}$ for all pairs $(s,t)\neq (i,j)$. Let $\mathfrak{R}:=\Z/p^M\Z[\Gp]$ denote the group ring with coefficients in $\Z/p^M\Z$. For each pair $(s,t)\neq (i,j)$ and $s\neq t$, may choose $g_{s,t}\in \Gp$ such that $p^m$ does not divide $\left(\gamma_{s,t}(g_{s,t})-\gamma_{i,j}(g_{s,t})\right)$. For each pair $(s,t)\neq (i,j)$ and $s\neq t$, let $h_{s,t}\in \mathfrak{R}$ denote the element $\gamma_{s,t}(g_{s,t})-g_{s,t}$. Let $g_0\in \Gp$ be such that $p^m$ does not divide $1-\gamma_{i,j}(g_{0})$, and set $h_0:=1-g_0$. 

\par Let $z\in \A$, we write $z$ as a linear combination $z=\sum_{k,l} a_{k,l} e_{k,l}$, where $a_{k,l}$ is the coefficient of $e_{k,l}$. The diagonal component of $z$ is $z_0=\sum_k a_{k,k} e_{k,k}$. We find that 
\[\begin{split}
& h_{s,t}z= \sum_{k,l} a_{k,l} \left(\gamma_{s,t}(g_{s,t})-\gamma_{k,l}(g_{s,t})\right) e_{k,l},\\
& h_0 z = \sum_{k,l}a_{k,l} \left(1-\gamma_{k,l}(g_{0})\right) e_{k,l}.
    \end{split}\]
Thus, applying $h_{s,t}$ to $z$ has an effect of killing the $e_{s,t}$-component and applying $h_0$ has the effect of killing the diagonal component $z_0$.
\par Set $H_{i,j}$ to denote the product $h_0\prod_{(s,t)} h_{s,t}$, where $(s,t)$ ranges over all pairs such that $(s,t)\neq (i,j)$ and $s\neq t$. Then, we set $x':=H_{i,j}x$. Clearly, $x'$ belongs to $\mathcal{M}$. From the discussion in the previous paragraph, we find that the $(s,t)$ component of $h_{s,t} x$ is $0$ for all pairs $(s,t)\neq (i,j)$, such that $s\neq t$, and the diagonal component of $h_0x$ is $0$. Therefore, $x'$ is a constant multiple of $e_{i,j}$. In fact, we find that 
\[x'=u \left(1-\gamma_{i,j}(g_{0})\right)\prod_{(s,t)} \left(\gamma_{s,t}(g_{s,t})-\gamma_{i,j}(g_{s,t})\right) e_{i,j},\]where the above product runs through all pairs $(s, t)\neq (i,j)$ such that $s\neq t$, and $u$ is a unit in $\Z/p^k\Z$. In fact $u$ is the $e_{i,j}$-coefficient of $x$, which is assumed to not be divisible by $p$. Since $\left(\gamma_{s,t}(g_{s,t})-\gamma_{i,j}(g_{s,t})\right)$ is not divisible by $p^m$ and there are $n^2-n-1$ pairs $(s,t)\neq (i,j)$ such that $s\neq t$, we find that the order of $p$ dividing $\prod_{(s,t)} \left(\gamma_{s,t}(g_{s,t})-\gamma_{i,j}(g_{s,t})\right)$ in the above product is $\leq m(n^2-n-1)$. Therefore, the power of $p$ dividing 
\[\left(1-\gamma_{i,j}(g_{0})\right)\prod_{(s,t)} \left(\gamma_{s,t}(g_{s,t})-\gamma_{i,j}(g_{s,t})\right)\] is at most $m(n^2-n)$. We therefore have shown that $p^{m(n^2-n)}e_{i,j}$ is also contained in $\mathcal{M}$, thus proving part \eqref{p1 of lemma 3.3}. The proof of part \eqref{p2 of lemma 3.3} is identical to that of \eqref{p1 of lemma 3.3}, hence, we omit it.
\end{proof}

\section{Controlling the image of Galois representations}\label{s 4}
\par Let $\rho:\Gp\rightarrow \op{GL}_n(\Z_p)$ be a Galois representation. In this section, we illustrate conditions for the image of $\rho$ to contain $\mathcal{U}_k$ for some large value of $k$.
Throughout, $\bar{\rho}$ is the residual representation of $\rho$, as defined in section \ref{s 2.2}. Recall that $\bar{\rho}$ is said to be $\{p\}$-unobstructed if $H^2(\op{G}_{\{p\}}, \op{Ad}\bar{\rho})=0$. Note that by Proposition \ref{prop unobstructed}, $\bar{\rho}$ is $\{p\}$-unobstructed.

\begin{proposition}\label{prop 4.1}
With respect to notation above, let \[\rho_N:\Gp\rightarrow \op{GL}_n(\Z/p^N\Z)\]be any lift of $\bar{\rho}$. Let $\tau:\Gp\rightarrow \op{GL}_1(\Z_p)$ be any character lifting $\op{det}\rho_N$. Then, there exists a lift $\rho:\Gp\rightarrow \op{GL}_n(\Z_p)$ of $\bar{\rho}$ such that \[\rho_N=\rho\mod{p^N}\text{ and }\op{det}\rho=\tau.\]
\end{proposition}
\begin{proof}
The above result follows from a well known argument, for further details, we refer to \cite[Fact 2.4 and Lemma 2.6]{ray2021constructing}.
\end{proof}
 Since there are infinitely many choices of $\tau$ that lift $\op{det}\rho_N$, it follows that there are infinitely many lifts $\rho$ of $\rho_N$.
 \par For $m\geq 1$, set $\mathcal{C}_{n,m}:=\op{ker}\left\{\op{GL}_n(\Z/p^{m+1})\rightarrow \op{GL}_n(\Z/p^m)\right\}$ let 
\begin{equation}\label{exp defn}\op{exp}_m: \op{M}_n(\F_p)\rightarrow \mathcal{C}_{n,m}\end{equation} be the map taking $A\in \op{M}_n(\F_p)$ to $\op{I}_n+p^m A$. Here, we define $p^m A$ by choosing a suitable lift of $A$ and then noting that the product $p^m A$ is independent of the choice of lift. It is easy to see that \eqref{exp defn} is an isomorphism of abelian groups. We let $\op{log}_m :\mathcal{C}_{n, m}\rightarrow \op{M}_n(\F_p)$ be the inverse of $\op{exp}_m$. Given a Galois representation $\varrho: \Gp\rightarrow \op{GL}_n(\Z/p^N\Z)$, and $i\leq N$ define a subspace $\Phi_i(\varrho)$ of $\op{M}_n(\F_p)$ as follows. Set $\varrho_i$ to denote the mod-$p^i$ reduction of $\varrho$, and set $L_{i}$ be the field $\Q(\varrho_{i})$. Consider the representation 
\[\varrho_i':=\varrho_{i|L_{i}}: \op{Gal}\left(\Q_{\{p\}}/L_{i}\right)\longrightarrow \op{GL}_n(\Z/p^{i+1}\Z).\] Note that the image of the above homomorphism is by design contained in $\mathcal{C}_{n, i}$. 

\begin{definition}\label{def of Phi}With respect to notation above, let $\Phi_i(\varrho)$ be the image of the composite $\op{log}_m \circ \varrho_i':\Gp\rightarrow M_n(\F_p)$. 
\end{definition}
\begin{lemma}Identify $\op{M}_n(\F_p)$ with $\op{Ad}\bar{\rho}$. With respect to notation above, then $\Phi_i(\varrho)$ is a $\op{G}_{\{p\}}$-submodule of $\op{Ad}\bar{\rho}$. 
\end{lemma}
\begin{proof}
Let $A$ be in $\Phi_i(\varrho)$, then, there exists $g'\in \op{Gal}\left(\Q_{\{p\}}/L_{i}\right)$ such that $\varrho_i'(g')=\op{I}_n +p^{i-1} A$. Let $g\in \Gp$, then, $gg'g^{-1}\in \op{Gal}\left(\Q_{\{p\}}/L_i\right)$. We find that 
\[\varrho_i'(gg'g^{-1})=\op{I}_n+p^{i-1} (g\cdot A).\]
Hence, $g\cdot A$ is also contained in $\Phi_i(\varrho)$.
\end{proof}
\begin{proposition}\label{prop bracket}
    With respect to notation above, such that $i+j\leq N$, then we find that $[\Phi_i(\varrho), \Phi_j(\varrho)]\subseteq \Phi_{i+j}(\varrho)$.
\end{proposition}
\begin{proof}
See the proof of \cite[Lemma 2.8]{ray2021constructing}. 
\end{proof}
The module $\op{Ad}\bar{\rho}$ is equipped with a Lie bracket $[X,Y]:=XY-YX$. Note that $\op{Ad}^0\bar{\rho}=[\op{Ad}^0\bar{\rho},\op{Ad}^0\bar{\rho}]$.
\begin{lemma}\label{lemma 4.4}
Suppose that $\rho:\Gp\rightarrow \op{GL}_n(\Z_p)$ is a Galois representation such that for some $i\geq 1$, $\Phi_i(\rho)$ contains $\op{Ad}^0\bar{\rho}$. Then, 
\begin{enumerate}
    \item for all $j\geq i$, $\Phi_j(\rho)$ contains $\op{Ad}^0\bar{\rho}$,
    \item the image of $\rho$ contains $\mathcal{U}_i$.
\end{enumerate}
\end{lemma}
\begin{proof}
The above result is \cite[Lemma 2.9]{ray2021constructing}.
\end{proof}
Let $\mu$ be the element $\mu=\sum_{i=1}^n a_i e_{i,i}\in \op{Ad}\bar{\rho}$, where $a_i=\frac{1+(-1)^i}{2}$ in $\mathfrak{t}$. 

\begin{proposition}\label{main prop}
Suppose that $\rho:\Gp\rightarrow \op{GL}_n(\Z_p)$ is a Galois representation such that for some $m\geq 1$, $\Phi_m(\rho)$ contains the following elements
\begin{enumerate}
    \item the element $\mu$ as defined above,
    \item $e_{i,j}$ for all pairs $(i,j)$ such that $i+j$ is odd.
\end{enumerate}
Then, the following assertions hold
\begin{enumerate}
    \item $\Phi_{4m}(\rho)$ contains $\op{Ad}^0\bar{\rho}$, 
    \item the image of $\rho$ contains $\mathcal{U}_{4m}$. 
\end{enumerate}
\end{proposition}
\begin{proof}
First, we assume that $n>2$. Let us show that $\Phi_{2m}(\rho)$ contains all $e_{i,j}$ for which $(i,j)$ is a pair such that $i\neq j$. Note that by Proposition \ref{prop bracket}, we have that $[\Phi_m(\rho), \Phi_m(\rho)]\subseteq \Phi_{2m}(\rho)$. Let $(i, j)$ be a pair such that $i+j$ is odd. Since $\mu$ and $e_{i,j}$ are both contained in $\Phi_m(\rho)$, it follows that $[\mu, e_{i,j}]\in \Phi_{2m}(\rho)$. Note that $[\mu, e_{i,j}]=(a_i-a_j) e_{i, j}$. Since $i\not \equiv j\mod{2}$, either $i$ is odd and $j$ is even, or, $i$ is even and $j$ is odd. We find that 
\[[\mu, e_{i,j}]=\begin{cases}
& e_{i,j}\text{ if }i\text{ is even and }j\text{ is odd,}\\
& -e_{i,j}\text{ if }j\text{ is even and }i\text{ is odd.}
\end{cases}\]
Therefore, we find that $e_{i,j}$ is contained in $\Phi_{2m}(\rho)$ for all pairs $(i,j)$ such that $i+j$ is odd. Now, let $(i,j)$ be a pair such that $i\neq j$ and $i+j$ is even, i.e., $i\equiv j\mod{2}$. Since $n\geq 3$, there exists $l$ such that $i+l$ and $j+l$ are both odd. Then, by assumption $e_{i,l}$ and $e_{l,j}$ are both contained in $\Phi_m(\rho)$. The relation $e_{i,j}=[e_{i,l}, e_{l,j}]$ implies that $e_{i,j}$ is contained in $\Phi_{2m}(\rho)$. Therefore, we have shown that $e_{i,j}$ is contained in $\Phi_{2m}(\rho)$ for all pairs $(i,j)$ such that $i\neq j$.
\par Note that $[\Phi_{m}(\rho), \Phi_{(r-1)m}(\rho)]$ is contained in $\Phi_{rm}(\rho)$ for $r\in \Z_{\geq 2}$. By induction on $r$, the same argument as above shows that for all $r\in \Z_{\geq 2}$, $\Phi_{rm}(\rho)$ contains $e_{i,j}$ for all pairs such that $i\neq j$. Since $[\Phi_{2m}(\rho), \Phi_{2m}(\rho)]$ is contained in $\Phi_{4m}(\rho)$, we find that $e_{i,i}-e_{j,j}=[e_{i,j}, e_{j,i}]$ is contained in $\Phi_{4m}(\rho)$ for all pairs $(i,j)$. Hence, we have shown that $\Phi_{4m}(\rho)$ contains $\op{Ad}^0\bar{\rho}$, the subspace of $\op{Ad}\bar{\rho}$ consisting of trace $0$ matrices. It follows from Lemma \ref{lemma 4.4} then the image of $\rho$ contains $\mathcal{U}_{4m}$. Therefore, we have proved the result for the case when $n>2$.

\par Next, we assume that $n=2$. In this case, a slightly different argument is required. Since $[\op{Ad}^0\bar{\rho}, \op{Ad}^0\bar{\rho}]=\op{Ad}^0\bar{\rho}$, it suffices to show that $\Phi_{2m}(\rho)$ contains $\op{Ad}^0\bar{\rho}$. Note that by assumption, $\Phi_m(\rho)$ contains $\mu=e_{1, 1}$, $e_{1,2}$ and $e_{2,1}$. Therefore, $\Phi_{2m}(\rho)$ contains 
\[\begin{split} & e_{1,2}=[e_{1,1}, e_{1,2}],\\
& e_{2,1}=[e_{2,1}, e_{1,1}],\\
& e_{1,1}-e_{2,2}=[e_{1,2}, e_{2,1}],
\end{split}
\]and therefore, contains $\op{Ad}^0\bar{\rho}$. Thus, the result follows from Lemma \ref{lemma 4.4}.
\end{proof}

\section{Proof of the main result}\label{s 5}

\par In this section, we shall prove the main result. We begin with a few preliminary Lemmas.
\begin{lemma}\label{lemma 5.1}
Suppose that $\psi:\Gp\rightarrow \left(\Z/p^k\Z\right)^\times $ is a character with mod $p$ reduction $\bar{\psi}$. Assume that all of the following conditions are satisfied
\begin{enumerate}
\item $\bar{\psi}$ is odd, i.e., $\bar{\psi}(c)=-1$, where $c$ denotes complex conjugation,
\item $H^0\left(\Gp, \F_p(\bar{\psi})\right)=0$,
\item $H^2\left(\Gp, \F_p(\bar{\psi})\right)=0$.
\end{enumerate}
Then, there is an isomorphism $H^1\left(\Gp, \Z/p^k\Z(\psi)\right)\simeq \Z/p^k\Z$.
\end{lemma}
\begin{proof}
First, we prove the result when $k=1$, in which case, $\psi$ is identified with its reduction $\bar{\psi}$. Since $\bar{\psi}$ is odd and $H^i\left(\Gp, \F_p(\bar{\psi})\right)=0$ for $i=0,2$, it follows from the global Euler characteristic formula (cf. \cite[Theorem 8.7.4]{neukirch2013cohomology}) that $H^1\left(\Gp, \F_p(\bar{\psi})\right)\simeq \Z/p\Z$. For $i\leq k$, let $\psi_i$ be the mod-$p^i$ reduction of $\psi$. Since $H^0(\Gp, \F_p(\bar{\psi}))=0$, it follows that $H^0\left(\Gp, \Z/p^{k-1}\Z(\psi_{k-1})\right)=0$. Consider the long exact sequence in cohomology associated with the following short exact sequence of Galois modules
\[0\rightarrow \F_p(\bar{\psi})\rightarrow \Z/p^k \Z(\psi)\rightarrow \Z/p^{k-1}\Z(\psi_{k-1})\rightarrow 0.\]Thus, we have the short exact sequence of cohomology groups
\begin{equation}\label{ses cohomology groups}0\rightarrow H^1\left(\Gp,\F_p(\bar{\psi})\right)\rightarrow H^1\left(\Gp, \Z/p^k\Z(\psi)\right)\rightarrow H^1\left(\Gp, \Z/p^{k-1}\Z(\psi_{k-1})\right)\rightarrow 0.\end{equation}
By induction, assume that the group $H^1\left(\Gp, \Z/p^{k-1}\Z(\psi_{k-1})\right)$ is isomorphic to $\Z/p^{k-1}\Z$. This shows that $H^1\left(\Gp, \Z/p^k\Z(\psi)\right)$ is either isomorphic to $\Z/p\Z\oplus \Z/p^{k-1}\Z$ or to $\Z/p^k\Z$. Since the last map is the mod-$p^{k-1}$ reduction map, it follows that $H^1\left(\Gp, \Z/p^k(\psi)\right)$ must be isomorphic to $\Z/p^k\Z$. This completes the proof.
\end{proof}

\par Recall that $\rho_M$ is the diagonal Galois representation $\rho_{\gamma}$ associated with an $m$-admissible $M$-lift $\gamma=(\gamma_1, \dots, \gamma_n)$ of $\alpha$. Let $L$ (resp. $L_0$) be the field extension $\Q(\rho_M)$ (resp. $\Q(\bar{\rho})$). Note that since $\rho_M$ is a diagonal representation, $L/\Q$ is an abelian extension. Set $\Delta:=\op{Gal}(L/\Q)$ (resp. $\Delta_0:=\op{Gal}(L_0/\Q)$). Let $\psi$ denote the character $\gamma_{i,j}:=\gamma_i \gamma_j^{-1}$ for a pair $(i,j)$ such that $i+j$ is odd. In this case, we note that $\bar{\psi}=\bar{\gamma}_{i,j}=\alpha_{i,j}$ and hence is odd. It follows from the argument in the proof of Proposition \ref{prop unobstructed} that $H^i\left(\Gp, \F_p(\bar{\psi})\right)=0$ for $i=0,2$. Recall from Lemma \ref{lemma 5.1} that $H^1(\Gp, \Z/p^M\Z (\psi))$ is isomorphic to $\Z/p^M\Z$ and let $f_{i,j}$ be a generator. The Galois action of $\op{Gal}(\Q_{\{p\}}/L)$ on $\Z/p^M \Z(\psi)$ is trivial. Therefore, $H^1\left(\op{Gal}(\Q_{\{p\}}/L), \Z/p^M\Z(\psi)\right)$ consists of homomorphisms. There is a natural action of $\Delta$ on  $\op{Hom}\left(\op{Gal}(\Q_{\{p\}}/L), \Z/p^M\Z(\psi)\right)$, which is described as follows. The restriction to $\op{Gal}(\Q_{\{p\}}/L)$ gives a homorphism \[\op{res}: H^1(\Gp, \Z/p^M\Z(\psi))\rightarrow \op{Hom}\left(\op{Gal}(\Q_{\{p\}}/L), \Z/p^M\Z(\psi)\right)^{\Delta}.\] Given $\sigma\in \Delta$, choose a lift $\tilde{\sigma}$ to $\Gp$. For $g\in \op{Gal}(\Q_{\{p\}}/L)$, observe that $\tilde{\sigma} g \tilde{\sigma}^{-1}$ is contained in $\op{Gal}(\Q_{\{p\}}/L)$, since $L/\Q$ is a Galois extension. Let $h$ be a homomorphism in $\op{Hom}\left(\op{Gal}(\Q_{\{p\}}/L), \Z/p^M\Z(\psi)\right)$. Note that the value $h(\tilde{\sigma} g \tilde{\sigma}^{-1})$ is independent of the choice of lift $\tilde{\sigma}$ of $\sigma$, since $h$ must factor through an abelian quotient of $\op{Gal}(\Q_{\{p\}}/L)$. The action of $\Delta$ is given by setting $(\sigma \cdot h)(g):=\psi(\sigma)^{-1} h(\tilde{\sigma} g \tilde{\sigma}^{-1})$. Therefore, the subgroup $\op{Hom}\left(\op{Gal}(\Q_{\{p\}}/L), \Z/p^M\Z(\psi)\right)^\Delta$ consists of those homomorphisms which satisfy the condition 
\begin{equation}\label{delta invariant}
h(\tilde{\sigma} g \tilde{\sigma}^{-1})=\psi(\sigma) h(g).
\end{equation}

\begin{lemma}\label{lemma 5.2}
With respect to notation above, the restriction map 
\begin{equation}\label{lemma 5.2 restriction map}\op{res}: H^1(\Gp, \Z/p^M\Z(\psi))\rightarrow \op{Hom}\left(\op{Gal}(\Q_{\{p\}}/L), \Z/p^M\Z(\psi)\right)^{\Delta}\end{equation} is injective. In other words, the homomorphism $\op{res}(f_{i,j})$ surjects onto $\Z/p^M\Z$.
\end{lemma}
\begin{proof}
Recall that $\psi=\gamma_{i,j}$ is a character modulo-$p^M$ such that $i+j$ is stipulated to be odd. Since $f_{i,j}$ is chosen to be a generator of $H^1(\Gp, \Z/p^M\Z(\psi))$, it is clear that the injectivity of \eqref{lemma 5.2 restriction map} is equivalent to the condition that $p^k\op{res}(f_{i,j})\neq 0$ for all $k<M$. This in turn is equivalent to the condition that the homomorphism $\op{res}(f_{i,j})$ is surjective. 

From the inflation-restriction sequence (cf. \cite{neukirch2013cohomology}), it suffices to show that \[H^1\left(\Delta, \Z/p^M\Z(\psi)\right)=0.\] We point out here that \[\left(\Z/p^M\Z(\psi)\right)^{\op{Gal}(\Q_{\{p\}}/L)}= \left(\Z/p^M\Z(\psi)\right)\] since $\psi$, when restricted to $\op{Gal}(\Q_{\{p\}}/L)$ is trivial. First, we show that $H^1(\Delta, \Z/p\Z(\bar{\psi}))=0$. Let $\mathcal{N}$ denote the Galois group $\op{Gal}(L/L_0)$. Note that $\mathcal{N}$ is an abelian group since $\Delta:=\op{Gal}(L/\Q)$ is abelian. Furthermore, $\mathcal{N}$ is a $p$-group. Consider the inflation-restriction sequence
\[0\rightarrow H^1\left(\Delta_0, \Z/p\Z(\bar{\psi})\right)\rightarrow H^1\left(\Delta, \Z/p\Z(\bar{\psi})\right)\rightarrow H^1\left(\mathcal{N}, \Z/p\Z(\bar{\psi})\right)^{\Delta_0}.\]Since $\Delta_0$ is isomorphic to the image of the diagonal representation $\bar{\rho}$, it is easy to see that $p\nmid \# \Delta_0$. Therefore, it follows (from a standard application of restriction-corestriction) that $H^1\left(\Delta_0, \Z/p\Z(\bar{\psi})\right)=0$. We show that $\op{Hom}\left(\mathcal{N}, \Z/p\Z(\bar{\psi})\right)^{\Delta_0}=0$, and thus deduce that $H^1\left(\Delta, \Z/p\Z(\bar{\psi})\right)=0$. The action of $\Delta_0$ on $\mathcal{N}$ is induced via conjugation. In greater detail, given $\sigma\in \mathcal{N}$ and $x\in \mathcal{N}$, we choose a lift $\tilde{\sigma}\in \Delta$ of $\sigma$, and set $\sigma\cdot x:=\tilde{\sigma} x \tilde{\sigma}^{-1}$. Since $\mathcal{N}$ is abelian, the action is well defined and $\tilde{\sigma} x \tilde{\sigma}^{-1}$ is independent of the choice of lift $\tilde{\sigma}$. Since $\Delta$ is abelian, this action is trivial, i.e., $\sigma\cdot x=x$. Note that $\bar{\psi}\neq 1$, and therefore, it follows that $\op{Hom}\left(\mathcal{N}, \Z/p\Z(\bar{\psi})\right)^{\Delta_0}=0$. We have thus shown that $H^1\left(\Delta, \Z/p\Z(\bar{\psi})\right)=0$.
 \par To conclude the proof, we show that for all $k\leq M$, $H^1\left(\Delta, \Z/p^i\Z(\psi_k)\right)=0$. Here, $\psi_k$ is the mod-$p^k$ reduction of $\psi$. The result is proven by induction on $k$. The case when $k=1$ has been proven. Consider the short exact sequence 
 \[0\rightarrow \Z/p\Z(\bar{\psi})\rightarrow \Z/p^k(\psi_k)\rightarrow \Z/p^{k-1}(\psi_{k-1})\rightarrow 0,\]  and the associated exact sequence
 \[H^1\left(\Delta,\Z/p\Z(\bar{\psi}) \right)\rightarrow H^1\left(\Delta, \Z/p^k(\psi_k) \right)\rightarrow H^1\left(\Delta,  \Z/p^{k-1}(\psi_{k-1})\right).\]
 If we assume that 
 \[H^1\left(\Delta,\Z/p\Z(\bar{\psi}) \right)=0\text{ and }H^1\left(\Delta,  \Z/p^{k-1}(\psi_{k-1})\right)=0, \] then, it follows that $H^1\left(\Delta, \Z/p^k(\psi_k) \right)=0$. Therefore, by induction, it follows that 
 \[H^1\left(\Delta, \Z/p^M(\psi) \right)=0.\]
\end{proof}
We make final preparations for the proof of the main theorem. First, we recall some context. Note that $p\geq 7$ is a prime number such that the Assumption \ref{main ass} holds. Let $\bar{\rho}$ be the representation $\op{diag}(\alpha_1, \dots, \alpha_n)$ as defined in section \ref{s 2.2}. Recall that Proposition \ref{prop unobstructed} asserts that $\bar{\rho}$ is $\{p\}$-unobstructed, i.e., $H^2(\Gp, \op{Ad}\bar{\rho})=0$. Recall that the numbers $m, M, N$ are defined as follows
\[m:=3+\lfloor \op{log}_p(2^n+1)\rfloor, M:=m(n^2-n)+1\text{ and }N:=2M.\]Recall that $\beta$ be an $m$-admissible $N$-lift of $\alpha$ and $\gamma$ is the mod-$p^M$ reduction of $\beta$. In particular, $\gamma$ is also $m$-admissible. Let $\rho_N:=\rho_\beta$ (resp. $\rho_M:=\rho_\gamma$). As explained previously, it follows from Lemma \ref{lemma 5.1} that for each pair $(i,j)$ such that $1\leq i,j\leq n$ and $i+j$ is odd, $H^1\left(\Gp, \Z/p^M\Z(\gamma_{i,j})\right)\simeq \Z/p^M\Z$. Let $f_{i,j}$ be a generator of $H^1\left(\Gp, \Z/p^M\Z(\gamma_{i,j})\right)$ for such a pair $(i,j)$. Let $\mathfrak{B}$ be the set of all pairs $(i,j)$ such that $1\leq i, j\leq n$ and $i+j$ is odd. Set $f=\sum_{(i,j)\in \mathfrak{B}} f_{i,j}$, the sum is taken over the generators $f_{i,j}$. This is viewed as an element in \[H^1(\Gp, \op{Ad}\rho_M)=\bigoplus_{(i,j)} H^1\left(\Gp, \Z/p^M\Z(\gamma_{i,j})\right).\] Let $\rho_N'$ be the twist \begin{equation}\label{def of rho N prime}(\op{I}_n+p^M f) \rho_N,\end{equation} as defined in section \ref{s 3}. The representation $\rho_N'$ is a deformation of $\rho_M$.
\par By Proposition \ref{prop unobstructed}, $\bar{\rho}$ is $\{p\}$-unobstructed. Therefore, it follows from Proposition \ref{prop 4.1}, $\rho_N'$ lifts to a characteristic zero lift $\varrho:\Gp\rightarrow \op{GL}_n(\Z_p)$. We recall that from the line prior to the statement of Proposition \ref{prop bracket} that $\mu=\sum_{i=1}^n a_i e_{i,i}$, where $a_i:=\frac{1+(-1)^i}{2}$. The field $L$ is the field $\Q(\rho_M)$ cut out by $\rho_M$. 
\par The Galois stable submodule $\mathcal{M}$ of $\op{Ad}\rho_M$ is defined in the paragraph above Proposition \ref{mathcal M prop}. The definition of $\Phi_N(\varrho)$ is given in Definition \ref{def of Phi}. 
\begin{proposition}\label{prop 5.3}
There is a natural Galois equivariant isomorphism of \[\Phi_N(\varrho) \simeq  p^{M-1}\mathcal{M}.\]
\end{proposition}
\begin{proof}
Consider the map 
\[\iota:\op{Ad}\bar{\rho}\rightarrow \op{Ad}\rho_M,\] which takes a matrix $A$ to $\iota(A):=p^{M-1} \tilde{A}$, where $\tilde{A}$ is any choice of lift of $A$ to $\op{Ad}\rho_M$. Clearly, $p^{M-1} \tilde{A}$ is independent of the choice of lift $\tilde{A}$, and we simply denote this product by $p^{M-1}A$. We show that $\iota $ induces the isomorphism $\Phi_N(\varrho) \xrightarrow{\sim}  p^{M-1}\mathcal{M}$. Let $A$ be a matrix in $\op{Ad}\rho_M$ and suppose that $A$ is contained in $\mathcal{M}$. This means that $\op{I}_n+p^{M} A$ is contained in the image of $\rho_N$. Now suppose that $A=p^{M-1} B$ for some matrix $B\in \op{Ad}\bar{\rho}$. Then, $\op{I}_n+p^{2M-1}B=\op{I}_n+p^{N-1}B$ is contained in the image of $\rho_N$. It thus follows that $B\in \Phi_N(\varrho)$. Conversely, for $B\in \Phi_N(\varrho)$, it is clear that $p^{M-1} B$ is in $\mathcal{M}$, by the same reasoning. The map $\iota$ is clearly Galois equivariant, and thus, induces the isomorphism $\Phi_N(\varrho) \xrightarrow{\sim}  p^{M-1}\mathcal{M}$.
\end{proof}

\begin{lemma}\label{last lemma}
With respect to notation above, $\Phi_N(\varrho)$ contains $\mu$ and all matrices $e_{i,j}$ such that $i+j$ is odd. 
\end{lemma}
\begin{proof}
Recall that for each pair $(i,j)\in \mathfrak{B}$, Lemma \ref{lemma 5.2} shows that the restricted homomorphism $f_{i,j}:\op{Gal}(\Q_{\{p\}}/L)\rightarrow \Z/p^M\Z$ is surjective. Therefore, for $(i,j)\in \mathfrak{B}$, there exists $g_{i,j}\in \op{Gal}(\Q_{\{p\}}/L)$ such that $f_{i,j}(g_{i,j})=1$. We write $\rho_N'(g_{i,j})=\op{I}_n+p^M x$. Then by construction the $e_{i,j}$-component of $x$ is $1$, and hence, by Lemma \ref{lemma 3.3}, $\mathcal{M}$ contains $p^{m(n^2-n)}e_{i,j}=p^{M-1} e_{i,j}$ for each pair $(i,j)\in \mathfrak{B}$. By Proposition \ref{prop 5.3}, $p^{M-1} \mathcal{M}=p^{m(n^2-n)}\mathcal{M}$ can be identified with $\Phi_N(\rho)$. We make this identification. It follows that for each pair $(i,j)\in \mathfrak{B}$, $\Phi_N(\rho)$ contains $e_{i,j}$.
\par Let $g$ in $\op{Gal}(\Q_{\{p\}}/L)$ be an element such that $\chi(g)=1+p^M$. Then, we find that $\rho_N(g)=\op{I}_n+\mu p^M$. Let $x$ be the element such that $\rho_N'(g)=\op{I}_n+p^Mx$. Since $f(g)$ consists of $0$s along the diagonal, it follows that the diagonal projection $x_0=\mu$. Lemma \ref{lemma 3.3} then implies that $\mathcal{M}$ contains $p^{m(n^2-n)}\mu$. Hence, By Proposition \ref{prop 5.3} implies that $\Phi_N(\varrho)$ contains $\mu$. 
\end{proof}

\begin{theorem}\label{main thm}
Let $p\geq 7$ be a prime and $M:=(n^2-n)\left(3+\lfloor \op{log}_p(2^n+1)\rfloor\right)+1$.Suppose that Assumption \ref{main ass} holds. Then, there are infinitely many Galois representations $\varrho:\Gp\rightarrow \op{GL}_n(\Z_p)$ whose image contains $\mathcal{U}_{8M}$. 
\end{theorem}
\begin{proof}
Recall the $N$ is set to be $2M$ and let $\varrho$ be as in the discussion above.  According to the Lemma \ref{last lemma}, $\Phi_N(\varrho)$ contains $\mu$ and $e_{i,j}$ for all $(i,j)$ such that $i+j$ is odd. It follows from Proposition \ref{main prop} that the image of $\varrho$ contains $\mathcal{U}_{4N}=\mathcal{U}_{8M}$. It follows from Proposition \ref{prop 4.1}, that there are infinitely many Galois representations $\varrho$ that lift $\rho_N'$, and hence, infinitely many Galois representations unramified outside $\{p, \infty\}$ and with image containing $\mathcal{U}_{8M}$.
\end{proof}
\bibliographystyle{alpha}
\bibliography{references}
\end{document}